\documentclass[3p]{elsarticle}

\usepackage[english,activeacute]{babel}
\usepackage[T1]{fontenc}
\usepackage{amsmath,amssymb}

\def\d{\hbox{d}}

\newtheorem{theorem}{Theorem}[section]

\newtheorem{proposition}{Proposition}[section]
\newtheorem{lemma}{Lemma}[section]

\newproof{proof}{Proof}
\newproof{dem-final}{Proof of Lemma \ref{lema-multiplicador}}
\newtheorem{notation}{Notation}

\begin{document}
\begin{frontmatter}
\title{An analytic approach to the ergodic theory of stochastic variational inequalities \footnote{This research was partially supported by a grant from CEA, Commissariat \`a l'\'energie atomique and by the National Science Foundation under grant DMS-0705247. A large part of this work was completed while one of the authors was visiting the University of Texas at Dallas and the Hong-Kong Polytechnic University. We wish to thank warmly these institutions for the hospitality and support.}}


\selectlanguage{english}
\author[authorlabel1]{Alain Bensoussan},
\ead{alain.bensoussan@utdallas.edu}
\author[authorlabel2]{Laurent Mertz}
\ead{mertz@ann.jussieu.fr}

\address[authorlabel1]{International Center for Decision and Risk Analysis, School of Management, University of Texas at Dallas, Box 830688, Richardson, Texas 75083-0688,
Graduate School of Business, the Hong Kong Polytechnic University,
Graduate Department of Financial Engineering, Ajou University.
This research in the paper was supported by WCU (World Class University) program through the National Research Foundation of Korea funded by the Ministry of Education, Science and Technology (R31 - 20007).}
\address[authorlabel2]{Universit\'e Pierre et Marie Curie-Paris 6,
Laboratoire Jaques Louis Lions, 4 place jussieu 75005 Paris.}
\begin{abstract}
{In an earlier work made by the first author with J. Turi (Degenerate Dirichlet Problems Related to the Invariant Measure of Elasto-Plastic Oscillators, AMO, 2008), the solution of a stochastic variational inequality modeling an elasto-perfectly-plastic oscillator has been studied. The existence and uniqueness of an invariant measure have been proven. Nonlocal problems have been introduced in this context. In this work, we present a new characterization of the invariant measure. The key finding is the connection between nonlocal PDEs and local PDEs which can be interpreted with short cycles of the Markov process solution of the stochastic variational inequality.
\vskip 0.5\baselineskip
\noindent{\bf R\'esum\'e}
\vskip 0.5\baselineskip
{\bf Une approche analytique de la th\'eorie ergodique des in\'equations variationnelles stochastiques.}
Dans un travail pr\'ec\'edent du premier auteur en collaboration avec Janos Turi (Degenerate Dirichlet Problems Related to the Invariant Measure of Elasto-Plastic Oscillators, AMO, 2008), la solution d'une in\'equation variationnelle stochastique mod\'elisant un oscillateur \'elastique-parfaitement-plastique a \'et\'e \'etudi\'ee. L'existence et l'unicit\'e d'une mesure invariante ont \'et\'e prouv\'ees. Des probl\`emes nonlocaux ont \'et\'e introduits dans ce contexte.
La conclusion importante est la connexion entre des EDPs nonlocales et des EDPs locales qui peuvent \^etre interpr\'et\'ees comme les cycles courts du processus de Markov solution de l'in\'equation variationnelle stochastique.}
\end{abstract}
\end{frontmatter}

\section*{Version fran\c{c}aise abr\'eg\'ee}
La dynamique de l'oscillateur \'elastique-parfaitement-plastique s'exprime \`a l'aide d'une \'equation \`a m\'emoire (voir (\ref{chap2:hysteresis1})-(\ref{chap2:hysteresis2})). A. Bensoussan et J. Turi ont montr\'e que la relation entre la vitesse et la composante \'elastique de l'oscillateur est un processus de Markov ergodique qui satisfait une in\'equation variationnelle stochastique (voir (\ref{chap2:svi})). La solution admet une mesure invariante caract\'eris\'ee par dualit\'e \`a l'aide d'une \'equation aux d\'eriv\'ees partielles avec des conditions de bord non-locales (voir (\ref{chap2:ulambda})). Dans ce travail, une nouvelle preuve de la th\'eorie ergodique est pr\'esent\'ee ainsi qu'une nouvelle caract\'erisation de l'unique distribution invariante. Dans ce contexte, nous d\'eduisons des nouvelles formules reliant des \'equations aux d\'eriv\'ees partielles avec des conditions de bord non-locales \`a des probl\`emes locaux (voir (\ref{chap2:link})). 
\selectlanguage{english}
\section{Introduction}
In the engineering literature, the dynamics of the elastic-perfectly-plastic (EPP) oscillator has been formulated as a process $x(t)$ which stands for the displacement of the oscillator, evolving with hysteresis. The evolution is defined by the problem
\begin{equation}
\label{chap2:hysteresis1}
\ddot{x} + c_0 \dot{x} + \mathbf{F}(x(s), 0 \leq s \leq t) = \dot{w}
\end{equation}
with initial conditions of displacement and velocity $x(0)=x, \quad \dot{x}(0) = y$.
Here $c_0>0$ is the viscous damping coefficient, $k>0$ the stiffness, $w$ is a Wiener process; $\mathbf{F}(x(s), 0 \leq s \leq t)$ is a nonlinear functional which depends on the entire trajectory $\{ x(s), 0 \leq s \leq t\}$ up to time $t$. The plastic deformation denoted by $\Delta(t)$ at time $t$ can be recovered from the pair $(x(t),\mathbf{F}(x(s), 0 \leq s \leq t))$ by the following relationship:
\begin{equation}\label{chap2:hysteresis2}
\mathbf{F}(x(s), 0 \leq s \leq t)= 
\left\{
\begin{array}{rcl}
k Y & if & x(t) = Y + \Delta(t),\\
k (x(t) - \Delta(t)) & if & x(t)  \in ]-Y+ \Delta(t),Y+ \Delta(t)[,\\
-k Y & if & x(t) = -Y + \Delta(t).\\
\end{array}
\right.
\end{equation}
where $Y$ is an elasto-plastic bound. Such elasto-plastic oscillator is simple and representative of the elasto-plastic behavior of a class of structure dominated by their first mode 
of vibration, they are employed to estimate prediction of failure of mechanical structures.
Karnopp \& Scharton \cite{KarSchar} proposed a separation between elastic states and plastic states and introduced a fictitious variable $z(t) := x(t) - \Delta(t)$.\\

Recently, the right mathematical framework of stochastic variational inequalities (SVI) modeling an EPP oscillator with noise has been introduced by one of the authors in \cite{BenTuri1}. 
Although SVI have been already studied in \cite{BenLio} to represent reflection-diffusion processes in convex sets, no connection with random vibration had been made so far.
The inequality governs the relationship between the velocity $y(t)$ and the variable $z(t)$:  
\begin{equation}\label{chap2:svi}
\d y(t)  = -(c_0 y(t) + k z(t)) \d t + \d w(t), \quad (\d z(t)-y(t)\d t)(\phi - z(t)) \geq 0, \quad \forall \vert  \phi \vert \leq Y, \quad \vert z(t) \vert \leq Y.
\end{equation}
Let us introduce some notations.
\begin{notation}
$D := \mathbb{R} \times (-Y,+Y), \quad D^+ := (0,\infty) \times \{Y \}, \quad D^- := (-\infty,0) \times \{-Y \}$, 
and the differential operators
$A \zeta :=    -\frac{1}{2} \zeta_{yy} + (c_0y+kz)  \zeta_{y}  - y \zeta_z, \quad B_+\zeta :=  -\frac{1}{2} \zeta_{yy} + (c_0y + kY)  \zeta_y, \quad 
B_-\zeta :=   -\frac{1}{2} \zeta_{yy} + (c_0y - kY)  \zeta_y.$
where $\zeta$ is a regular function on $D$.
\end{notation}
In \cite{BenTuri1}, it has been shown that the probability distribution of $(y(t),z(t))$ converges to an asymptotic probability measure on  $D \cup D^+ \cup D^-$ namely $\nu$. Moreover, $\nu$ is the unique invariant distribution of $(y(t),z(t))$. 
In addition, from \cite{BenTuri2} we know also that there exists a unique solution $u_\lambda$ to the following partial differential equation (PDE)
\begin{equation}
\lambda u_{\lambda} + A u_{\lambda} =  f \quad \textup{in} \quad D, \quad \lambda u_{\lambda} + B_+ u_{\lambda}  =   f \quad \textup{in} \quad D^+ , \quad \lambda u_{\lambda} + B_- u_{\lambda}  =  f  \quad \textup{in} \quad D^-\\
\label{chap2:ulambda}
\end{equation}
with the nonlocal boundary conditions given by the fact that $u_{\lambda}(y,Y)$ and $u_{\lambda}(y,-Y)$ are continuous, where $\lambda >0$ and $f$ is a bounded measurable function.
The function $u_\lambda$ satisfies $\| u_\lambda \|_{\infty} \leq \frac{\| f \|_{\infty}}{\lambda}$, $u_{\lambda}$ is continuous
and for all  $(\eta, \zeta) \in \bar{D}$, we have $\lim_{\lambda \to 0 } \lambda u_\lambda (\eta, \zeta) = \nu(f)$. 
We use the notation $u_\lambda(y,z;f)$.\\

Now, we introduce the short cycles to provide a new proof of the ergodic theory for (\ref{chap2:svi}). In this context, we derive new formulas linking PDEs with nonlocal boundary conditions to local problems.

\subsection{Short cycles}
Let $\lambda >0$, consider $v_{\lambda}(y,z)$ the solution of 
\begin{equation}
\lambda v_{\lambda} + A v_{\lambda} =  f \quad \textup{in} \quad D, \quad \lambda v_{\lambda} + B_+ v_{\lambda}  =  f  \quad \textup{in} \quad D^+ , \quad \lambda v_{\lambda} + B_- v_{\lambda}  = f  \quad \textup{in} \quad D^-\\
\label{chap2:vlambda}
\end{equation}
with the local boundary conditions $v_{\lambda}(0^+,Y)=0$ and $v_{\lambda}(0^-,-Y)=0$.
Also, if $f$ is symmetric (resp. antisymmetric) then $v_\lambda$ is symmetric (resp. antisymmetric). We use the notation $v_\lambda(y,z;f)$.
As $\lambda \to 0$, $v_\lambda \to v$ with
\begin{equation}\tag{$P_v$}
 A v =  f \quad \textup{in} \quad D, \quad  B_+ v  =  f \quad \textup{in} \quad D^+ , \quad  B_- v  =  f \quad \textup{in} \quad D^-\\
\label{chap2:v}
\end{equation}
with the local boundary conditions $v(0^+,Y)=0$ and $v(0^-,-Y)=0$.
We use the notation $v(y,z;f)$. We call $v(y,z;f)$ \underline{a short cycle}.
We detail the solution of \eqref{chap2:v} in the next section.
We introduce next $\pi^+(y,z)$ and $\pi^-(y,z)$ such that
\begin{equation}
A \pi^+  =  0 \quad \textup{in} \quad D, \quad \pi^+ = 1 \quad \textup{in} \quad D^+, \quad \pi^+  = 0 \quad \textup{in} \quad D^-
\label{chap2:piplus}
\end{equation}
and
\begin{equation}
A \pi^- = 0  \quad \textup{in} \quad D, \quad \pi^- = 0  \quad \textup{in} \quad D^+, \quad \pi^-  = 1 \quad \textup{in} \quad D^-.
\label{chap2:pimoins}
\end{equation}
We have $\pi^+ + \pi^- = 1$, so the existence and uniqueness of a bounded solution to (\ref{chap2:piplus}) and (\ref{chap2:pimoins}) are clear. 
A new formulation of the invariant distribution is given by the following theorem.
\begin{theorem}[New formulation of the invariant distribution $\nu$]
\label{chap2:th1}
Let $f$ be a bounded measurable function on $\bar{D}$, we have the following analytical characterization of the invariant distribution:
\[
\nu(f) = \frac{v(0^-,Y;f)+ v(0^+,-Y;f)}{2v(0^+,Y;1)}.
\]
Denote $\nu_\lambda(f) := \frac{v_\lambda(0^-,Y;f) + v_\lambda(0^+,-Y;f)}{2v_\lambda(0^-,Y;1)}$. 
As $\lambda \to 0$,
\begin{equation} \label{chap2:e23}
 u_{\lambda}(y,z;f) - \frac{\nu_{\lambda}(f)}{\lambda} \to u(y,z;f), \quad  \nu_\lambda(f) \to \nu(f)
\end{equation}
where $u$ satisfies
\begin{equation}
A u  =  f - \nu(f) \quad \textup{in} \quad D, \quad B_+u \quad = \quad f - \nu(f) \quad \textup{in} \quad D^+, \quad B_-u  \quad  = \quad f - \nu(f) \quad \textup{in} \quad D^-
\label{chap2:e24}
\end{equation}
with the nonlocal boundary conditions given by the fact that
\[
u(y,Y) \quad \mbox{and} \quad u(y,-Y) \quad \mbox{ are continuous}.
\]
Then, we obtain also the representation formula 
\begin{equation}\label{chap2:link}
u(y,z;f) =  v(y,z;f) - \nu(f) v(y,z;1) + \frac{\pi^+(y,z)-\pi^-(y,z)}{4 \pi^-(0^-,Y)} (v(0^-,Y;f)-v(0^+,-Y;f))
\end{equation}
\end{theorem}

\section{Analysis of the short cycles}
We describe the solution of \eqref{chap2:v}.
We can write $v(y,z;f) = v_e(y,z;f) + v^+(y,z;f) + v^-(y,z;f)$ with $v_e, v^+,v^-$ satisfying
\begin{equation}
A v_e =  f(y,z) \quad  \textup{in} \quad D, \quad v_e  =  0  \quad \textup{in} \quad D^+, \quad v_e    = 0 \quad \textup{in} \quad D^-,
\label{chap2:ve}
\end{equation}
\begin{equation}
A v^+  =  0  \quad \textup{in} \quad D, \quad v^+(y,Y)  =  \varphi^+(y;f)  \quad \textup{in} \quad D^+,\quad v^+    =  0 \quad \textup{in} \quad D^-,
\label{chap2:vplus}
\end{equation}
and
\begin{equation}
A v^-  =  0   \quad \textup{in}  \quad D, \quad v^-  =  0   \quad \textup{in}  \quad D^+, \quad v^-(y,-Y)  =  \varphi^-(y;f)  \quad \textup{in}  \quad D^-.
\label{chap2:vmoins}
\end{equation}
where $\varphi^+(y;f)$ and $\varphi^-(y;f)$ are defined by
\begin{equation}\label{chap2:phiplus}
-\frac{1}{2} \varphi_{yy}^+ + (c_0y+kY)\varphi_y^+  = f(y,Y), \quad y >0, \quad \varphi^+(0^+;f)=0
\end{equation}
and
\begin{equation}\label{chap2:phimoins}
-\frac{1}{2} \varphi_{yy}^- + (c_0y-kY)\varphi_y^-  = f(y,-Y), \quad y <0, \quad \varphi^-(0^-;f)=0.
\end{equation}
We check easily the formula $\varphi^+(y;f) = 2 \int_0^{\infty} d \xi \exp(-(c_0 \xi^2 + 2 kY \xi)) \int_{\xi}^{\xi + y} f(\zeta;Y) \exp(-2 c_0 \xi (\zeta - \xi)) d \zeta$, if  $y \geq 0$
and also $\varphi_-(y;f) = 2 \int_0^{\infty} d \xi \exp(-(c_0 \xi^2 - 2 kY \xi)) \int_{y-\xi}^{-\xi} f(\zeta;-Y) \exp(-2 c_0 \xi (\zeta - \xi)) d \zeta$, if $y \leq 0$.
\subsection{Solution to Problem \eqref{chap2:ve}}
The proof will be based on solving a sequence of Interior Exterior Dirichlet problems and a fixed point argument. Thus, we need to state the two following lemmas as preliminary results.
It is sufficient to consider $f =1$, with no loss of generality.
\subsubsection{Interior Dirichlet problem}
We begin with the interior problem, let  $D_1 := (-\bar{y}_1,\bar{y}_1) \times (-Y,Y), \quad D_1^+ := [0,\bar{y}_1) \times \{ Y \}, \quad D_1^- := (-\bar{y}_1,0] \times \{ -Y \}$.
Let us consider the space $C_1^+$ of continuous functions on $[-Y,Y]$ which are $0$ on $Y$ and the space $C_1^-$ of continuous functions on $[-Y,Y]$ which are $0$ on $-Y$.
Let $\varphi^+ \in C_1^+$ and $\varphi^- \in C_1^-$. We consider the problem
\begin{equation}
\label{chap2:interior}
-\frac{1}{2} \zeta_{yy} + (c_0 y + kz) \zeta_y - y \zeta_z = 1 \quad \mbox{in} \quad D_1, \quad \zeta(y,Y) =  0 \quad \mbox{in} \quad D_1^+, \quad \zeta(y,-Y)  =  0\quad \mbox{in} \quad D_1^-
\end{equation}
with $\zeta(\bar{y}_1,z) =  \varphi^+(z)$ and $\zeta(-\bar{y}_1,z)  =  \varphi^-(z)$, if  $-Y < z < Y.$
\begin{lemma}
There exists a unique bounded solution to the equation \eqref{chap2:interior}.
\end{lemma}
\begin{proof}
It is sufficient to prove an a priori bound. For that we can assume $\varphi^+, \varphi^-=0$. Consider $\lambda>0$ and the function
$\theta(y,z) = \exp(\lambda c_0(y^2 + kz^2))$ then $-\frac{1}{2} \theta_{yy} + (c_0y + kz) \theta_y - y\theta_z = \theta \left ( -\lambda c_0 + 2 \lambda c_0^2 y^2 (1-\lambda) \right )$.
Set next $H := - (\theta + \zeta)$ then 
\begin{equation}
\label{chap2:H}
-\frac{1}{2} H_{yy} + (c_0y + kz) H_y - y H_z = -1 + \theta \left ( \lambda c_0 - 2 \lambda c_0^2 y^2 (1-\lambda) \right ).
\end{equation}
If we pick $\lambda > \max(1,\frac{1}{c_0})$ the right hand side of \eqref{chap2:H} is positive. Therefore the minimum of $H$ can occur only on the boundary $y = \bar{y}_1$ and $z=Y$ with $y>0$ or $z=-Y$ with $y<0$. It follows that $H(y,z) \geq -\exp(\lambda c_0 (\bar{y}_1^2 + Y^2))$
and thus also
$0 \leq \zeta \leq \exp(\lambda c_0 (\bar{y}_1^2 + Y^2))$.
\end{proof}
\subsubsection{Exterior Dirichlet problems}
Now, we proceed by considering two exterior Dirichlet problems.
Let $0< \bar{y} < \bar{y}_1$, we define $D_{\bar{y}<y} := \{ y > \bar{y}, \quad -Y < z < Y\}$, $D_{\bar{y}<y}^+ := \{ y > \bar{y}, \quad z = Y\}$ and $D_{y<-\bar{y}} := \{ y < -\bar{y}, \quad -Y < z < Y\}$, $D_{y<-\bar{y}}^- := \{ y < -\bar{y}, \quad z = -Y\}$
and consider
\begin{equation}\label{chap2:exteriorp}
-\frac{1}{2} \eta_{yy}^+ + (c_0 y + kz) \eta_y^+ - y \eta_z^+ = 1 \quad \mbox{in} \quad D_{\bar{y}<y}, \quad \eta^+(y,Y)  =   0 \quad \mbox{in} \quad D_{\bar{y}<y}^+
\end{equation}
with the condition $\eta^+(\bar{y},z) = \zeta(\bar{y},z)$ if $-Y < z <Y$, and
\begin{equation}\label{chap2:exteriorm}
-\frac{1}{2} \eta_{yy}^- + (c_0 y + kz) \eta_y^- - y \eta_z^- = 1 \quad \mbox{in} \quad D_{y<-\bar{y}}, \quad \eta^-(y,-Y)  =   0 \quad \mbox{in} \quad D_{y<-\bar{y}}^-
\end{equation}
with the condition $\eta^-(-\bar{y},z) = \zeta(-\bar{y},z)$, if $-Y < z <Y$.
We use the same notation $\eta(y,z)$ for the two problems \eqref{chap2:exteriorp},\eqref{chap2:exteriorm} for the convenience of the reader. We have
\begin{lemma}
For any $\bar{y}>0$ there exists a unique bounded solution of (\ref{chap2:exteriorp}),(\ref{chap2:exteriorm}). 
\end{lemma}
\begin{proof}
It is sufficient to prove the bound, we claim that  $\| \zeta \|_{\infty} \leq \eta(y,z) \leq \| \zeta \|_{\infty} + \frac{Y-z}{\bar{y}}$, for $y > \bar{y}$
and
$\| \zeta \|_{\infty} \leq \eta(y,z) \leq \| \zeta \|_{\infty} + \frac{Y+z}{\bar{y}}$, for $y < -\bar{y}$.
Consider for instance $\rho(z) = \| \zeta \|_{\infty} + \frac{Y-z}{\bar{y}}$ for  $y > \bar{y}, -Y< z < Y$ then $-\frac{1}{2} \rho_{yy} + \rho_y (c_0 y + kz) - y \rho_z = \frac{y}{\bar{y}} > 1$, $\rho(\bar{y},z) = \| \zeta \|_{L^\infty} + \frac{Y-z}{\bar{y}} > \zeta(\bar{y},z)$, $\rho(\bar{y},z) = \| \zeta \|_{L^\infty} > 0$.
So clearly $\eta(y,z) \leq \rho(z)$.
So in all cases we can assert that $\| \eta \|_{\infty} \leq \| \zeta \|_{\infty} + \frac{2Y}{\bar{y}}$.
\end{proof}
\subsubsection{Solution to Problem \eqref{chap2:ve}}
\begin{proposition}\label{chap2:p1}
There exists a unique bounded solution to Problem \eqref{chap2:ve}.
\end{proposition}
\begin{proof}
Uniqueness comes from maximum principle.
Setting $\Phi = (\varphi^+(z),\varphi^-(z))$ and using the notation $\Phi(\bar{y}_1,z) = \varphi^+(z), \quad \Phi(-\bar{y}_1,z) = \varphi^-(z)$,
we can next define $\Gamma \Phi(\bar{y}_1,z) = \eta(\bar{y}_1,z)$ and $\Gamma \Phi(-\bar{y}_1,z) = \eta(-\bar{y}_1,z)$.
We thus have defined a map $\Gamma$ from $C_1^+,C_1^-$ into itself.
If $\Gamma$ has a fixed point, then it is clear that the function
\[
v_e(y,z) = 
\left\{
\begin{array}{l}
\zeta(y,z), \quad -\bar{y}_1 < y < \bar{y}_1,\\
\eta(y,z), \quad y > \bar{y}, \quad y < -\bar{y} 
\end{array}
\right.
\]
is a solution of \eqref{chap2:ve} since $\zeta = \eta \quad \mbox{for} \quad \bar{y} < y < \bar{y}_1, z \in (-Y,Y) \quad \mbox{and for} \quad -\bar{y}_1 < y < -\bar{y}, z \in (-Y,Y)$  
and the required regularity is available at boundary points $\bar{y},\bar{y}_1, -\bar{y}, - \bar{y}_1$. The result will follow from the property : $\Gamma$ is a contraction mapping.
This property will be an easy consequence of the following result. Consider the exterior problem
\begin{equation}\label{chap2:psi}
-\frac{1}{2} \psi_{yy} + \psi_y(c_0y + kz) - y \psi_z = 0 \quad \mbox{in} \quad D_{\bar{y}<y}, \quad
\psi(y,Y) =  0 \quad  \mbox{in} \quad D_{\bar{y}<y}^+, 
\end{equation}
where $\psi(\bar{y},z) = 1 \quad \mbox{if} \quad -Y<z<Y$, then $\sup_{-Y< z < Y} \psi(\bar{y}_1, z) \leq \rho < 1$.

Indeed if $\sup_{ -Y< z < Y } \psi(\bar{y}_1, z) =1$, then the maximum is attained on the line $y=\bar{y}_1$, and this is impossible because it cannot be at $z=Y$, nor at $z=-Y$, nor at the interior, by maximum principle considerations.
\end{proof}
\subsection{Solution to Problems \eqref{chap2:vplus} and \eqref{chap2:vmoins}}
We now consider the function $\varphi^+$ and $\varphi^-$ solution of \eqref{chap2:phiplus} and \eqref{chap2:phimoins}. 
Note that if $y<0$, we have $\varphi^-(y;1)=\varphi^+(-y;1)$. So it is sufficient to consider \eqref{chap2:phiplus} and we easily see that
\[
\varphi^+(y;1) = \int_0^{\infty} \exp(-(c_0 \xi^2 + 2 kY \xi)) \frac{1-\exp(-2c_0y\xi)}{2 c_0 \xi} \d \xi, \quad  \mbox{if} \quad  y >0
\]
and we have
$\varphi^+(y;1) \leq \frac{1}{c_0} \log(\frac{c_0 y + kY}{kY}), \quad \mbox{if} \quad y >0$.
We next want to solve the problem \eqref{chap2:vplus}. 
We proceed as follow. We extend $\varphi^+$ for $y<0$, by a function which is $C^2$ on $\mathbb{R}$ and with compact support on $y<0$. It is convenient to call $\varphi(y)$ the $C^2$ function on $\mathbb{R}$, with compact support for $y<0$ and $\varphi(y)=\varphi^+(y;1)$ for $y>0$. We set 
$w^+(y,z) = v^+(y,z) - \varphi(y)$
then we obtain the problem
\begin{equation}\label{chap2:wplus}
A w^+  =  g \quad  \mbox{in} \quad D, \quad 
w^+(y,Y) = 0, \quad  \mbox{in} \quad D^+, \quad
w^+(y,-Y) = -\varphi(y) , \quad  \mbox{in} \quad D^-
\end{equation}
with
$g(y,z) = - \left ( -\frac{1}{2}\varphi_{yy} + (c_0y+kz) \varphi_y \right )$.
\\\\
But, 
$g(y,z) = \mathbf{1}_{\{ y > 0 \}} \left (-1 + k(Y-z) \varphi_y(y) \right )$ + $\mathbf{1}_{\{ y < 0 \}} \left (-(-\frac{1}{2} \varphi_{yy} + (c_0 y + kz)\varphi_y ) \right )$ 
and thus, taking into account the definition of $\varphi$ when $y<0$, we can assert that $g(y,z)$ is a bounded function. Again, from the definition of $\varphi(y)$ when $y<0$, we obtain that on the boundary, $w^+$ is bounded. It follows from what was done for Problem \eqref{chap2:ve} that \eqref{chap2:wplus} has a unique solution. So we can state the following proposition.

\begin{proposition}\label{chap2:p2}
There exists a unique solution to \eqref{chap2:vplus} of the form $v^+(y,z) = \varphi^+(y) \mathbf{1}_{\{ y > 0 \}} + \tilde{v}^+(y,z)$
where  $\tilde{v}^+(y,z)$ is bounded. Similarly, there exists a unique solution to \eqref{chap2:vmoins} of the form $v^-(y,z) = \varphi^-(y) \mathbf{1}_{\{ y < 0 \}} + \tilde{v}^-(y,z)$
where  $\tilde{v}^-(y,z)$ is bounded.
\end{proposition}
\begin{proof}
We just define $\varphi(y)$ extension of $\varphi^+(y)$ for $y<0$ as explained before and consider $w^+(y,z)$ solution of \eqref{chap2:wplus}. 
We know that $w^+(y,z)$ is bounded and we have $v^+(y,z) = \varphi(y) + w^+(y,z) = \varphi^+(y) \mathbf{1}_{\{ y >0 \}} + \varphi(y) \mathbf{1}_{\{ y < 0 \}} + w^+(y,z)$
which is of the form \eqref{chap2:vplus} with $\tilde{v}^+(y,z) = \varphi(y) \mathbf{1}_{\{ y<0 \}} + w^+(y,z)$.
\end{proof}
\subsection{The complete Problem \eqref{chap2:v}}
Finally, we consider the complete Problem \eqref{chap2:v}, we can state 
\begin{theorem}
There exists a unique solution of \eqref{chap2:v} of the form $v(y,z;f) = \varphi^+(y;f) \mathbf{1}_{\{ y > 0 \}} + \varphi^-(y;f) \mathbf{1}_{\{ y < 0 \}} + \tilde{w}(y,z)$
where $\tilde{w}(y,z)$ is a bounded function which can be written as $\tilde{w}= v_e + w^+ + w^-$.
\end{theorem}
\begin{proof}
We just collect the results of Propositions \ref{chap2:p1} and \ref{chap2:p2}.
\end{proof}

\section{Ergodic Theorem}
\label{mainresult}
\begin{proof}[Proof of Theorem \ref{chap2:th1}]
We first prove the result when $f$ is symmetric. In that case, we can write
\begin{equation}\label{chap2:e25}
u_\lambda(y,z;f) = v_\lambda(y,z;f) + \frac{v_\lambda(0^-,Y;f)}{v_\lambda(0^-,Y;1)} \left ( \frac{1}{\lambda} - v_\lambda(y,z;1) \right )
\end{equation}
Indeed, we know that $u_\lambda(y,z;f)$ and  $v_\lambda(y,z;f)$ are symmetric. Setting $\tilde{u}_{\lambda}(y,z;f) = u_{\lambda}(y,z;f) - v_{\lambda}(y,z;f)$,  
we obtain
\begin{equation}
\lambda \tilde{u}_{\lambda} + A \tilde{u}_{\lambda}    = 0 \quad \textup{in} \quad D, \quad 
\lambda \tilde{u}_{\lambda} + B_+ \tilde{u}_{\lambda} = 0 \quad \textup{in} \quad D^+,\quad
\lambda \tilde{u}_{\lambda} + B_- \tilde{u}_{\lambda}  = 0 \quad \textup{in} \quad D^-
\label{utilde}
\end{equation}
with the boundary conditions $\tilde{u}_\lambda(0^+,Y;f) - \tilde{u}_\lambda(0^-,Y;f) = v_\lambda(0^-,Y;f)$ and $\tilde{u}_\lambda(0^+,-Y;f) - \tilde{u}_\lambda(0^-,-Y;f) = -v_\lambda(0^+,-Y;f)$.
This last condition is automatically satisfied, thanks to the previous one and the symmetry. The function $\frac{1}{\lambda}-v_\lambda(y,z;1)$ satisfies the three partial differential equations on $D$, $D^+$ and $D^-$. So,  $\tilde{u}_\lambda = C \left ( \frac{1}{\lambda} -v_\lambda(y,z;1) \right )$
and writing the first boundary condition, we have $\tilde{u}_\lambda(0^+,Y;f) - \tilde{u}_\lambda(0^-,Y;f) = -C \left ( v_\lambda(0^+,Y;1) - v_\lambda(0^-,Y;1) \right ) =  C v_\lambda(0^-,Y;1)$.
Hence, $C = \frac{v_\lambda(0^-,Y;f) }{v_\lambda(0^-,Y;1)}$ and formula \eqref{chap2:e25} has been obtained.
Now, we have $\nu_\lambda(f) \to \nu(f)= \frac{v(0^-,Y;f)}{v(0^-,Y;1)}$, as $\lambda \to 0$. 
If we define $u_\lambda^\star(y,z;f) = u_\lambda(y,z;f) - \frac{\nu_\lambda(f)}{\lambda} = v_\lambda(y,z;f) - \nu_\lambda(f)v_\lambda(y,z;1)$.
The function $u_\lambda^\star(y,z;f) \to v(y,z;f) - \nu(f)v(y,z;1) = v(y,z;f - \nu(f)), \quad \lambda \to 0$.
Also from its definition the function $u_\lambda^\star(y,Y;f)$ and $u_\lambda^\star(y,-Y;f)$ are continuous. From the choice of $\nu(f)$ the function
$v(y,Y; f - \nu(f))$ is continuous.
Now, since $f - \nu(f)$ is symmetric  $v(0^+,-Y; f - \nu(f)) - v(0^-,-Y; f - \nu(f)) = v(0^+,Y; f - \nu(f)) - v(0^-,Y; f - \nu(f)) = 0$. 
So the result is proven when $f$ is symmetric. We now consider the situation when $f$ is antisymmetric. We know that $u_\lambda(y,z;f)$ is antisymmetric. Similarly $v_\lambda(y,z;f)$ is antisymmetric. Consider $\pi_\lambda^-$ and $\pi_\lambda^+$ defined by
\begin{equation}
\lambda \pi_\lambda^+ + A \pi_\lambda^+ = 0 \quad \textup{in} \quad D,\quad
\lambda \pi_\lambda^+  + B_+ \pi_\lambda^+  = 0 \quad \textup{in} \quad D^+,\quad
\pi_\lambda^+  = 0 \quad \textup{in} \quad D^-
\label{chap2:pilambdaplus}
\end{equation}
with the boundary condition $\pi_\lambda^+(0^+,Y)=1$
and
\begin{equation}
\lambda \pi_\lambda^- + A \pi_\lambda^- = 0  \quad \textup{in} \quad D,\quad
\pi_\lambda^- = 0  \quad \textup{in} \quad D^+,\quad
\lambda \pi_\lambda^-  + B_+ \pi_\lambda^-  = 0 \quad \textup{in} \quad D^-
\label{chap2:pilambdamoins}
\end{equation}
with the boundary condition $\pi_\lambda^-(0^-,-Y)=1$.
We have $\pi_\lambda^-(y,z) = \pi_\lambda^-(-y,-z)$, we then state the formula $u_\lambda(y,z;f) = v_\lambda(y,z;f) - \frac{\left (\pi_\lambda^+(y,z)-\pi_\lambda^-(y,z) \right ) v_\lambda(0^+,-Y;f)}{1-\pi_\lambda^+(0^-,Y) + \pi_\lambda^+(0^+,-Y)}$.
So we see that $u_\lambda(y,z;f)$ converges as $\lambda \to 0$,without substracting a number $\frac{\nu_\lambda(f)}{\lambda}$. 
The function $u_\lambda(y,z;f)$ converges pointwise to $u(y,z;f) = v(y,z;f) - \frac{\left (\pi^+(y,z)-\pi^-(y,z) \right ) v(0^+,-Y;f)}{2\pi^-(0^-,Y)}$.
So when $f$ is antisymmetric, the results \eqref{chap2:e23}-\eqref{chap2:e24} hold with $\nu_\lambda(f)=0$ and $\nu(f)=0$.
For the general case, we can write $f = f_{\textup{sym}} + f_{\textup{asym}}$ with $f_{\textup{sym}}(y,z) = \frac{f(y,z)+f(-y,-z)}{2}, \quad f_{\textup{asym}}(y,z) = \frac{f(y,z)-f(-y,-z)}{2}$.
We have $\nu(f_\textup{sym})  =  \frac{v(0^-,Y; f_\textup{sym})}{v(0^-,Y;1)}$
and thus $\nu(f_\textup{sym})  =  \frac{v(0^-,Y; f) + v(0^+,-Y; f)}{2v(0^+,-Y;1)}$
Since $\nu(f_\textup{asym}) = 0$,
we deduce
$\nu(f) = \nu(f_\textup{sym})  =  \frac{v(0^-,Y; f) + v(0^+,-Y; f)}{2v(0^+,-Y;1)}$.
We obtain also the representation formula 
$u(y,z;f) = v(y,z;f) - \nu(f) v(y,z;1) + \frac{\pi^+(y,z)-\pi^-(y,z)}{4 \pi^-(0^-,Y)} (v(0^-,Y;f)-v(0^+,-Y;f))$
and the result is obtained.
\end{proof}

\end{document}